\documentclass[11pt]{amsart}
\usepackage{tgpagella}
\usepackage{euler}
\usepackage[T1]{fontenc}
\usepackage{amsmath, amssymb}
\usepackage[hidelinks]{hyperref}
\usepackage[english]{babel}
\usepackage{mathrsfs}
\usepackage{eucal}
\usepackage[all]{xy}
\newtheorem{thm}{Theorem}[section]
\newtheorem*{theorem}{Theorem}
\newtheorem{defn}[thm]{Definition}
\newtheorem{lemma}[thm]{Lemma}
\newtheorem{prop}[thm]{Proposition}
\newtheorem{cor}[thm]{Corollary}
\newtheorem{rmk}[thm]{Remark}

\newtheorem{fact}[thm]{Fact}
\newtheorem{question}[thm]{Question}
\newcommand\e\epsilon

\makeatletter

\def\indsym#1#2{%
  \setbox0=\hbox{$\m@th#1x$}%
  \kern\wd0%
  \hbox to 0pt{\hss$\m@th#1\mid$\hbox to 0pt{$\m@th#1^{#2}$}\hss}%
  \lower.9\ht0\hbox to 0pt{\hss$\m@th#1\smile$\hss}%
  \kern\wd0}

\def\nindsym#1#2{%
  \setbox0=\hbox{$\m@th#1x$}%
  \kern\wd0%
  \hbox to 0pt{\hss$\m@th#1\not$\kern1.4\wd0\hss}
  \hbox to 0pt{\hss$\m@th#1\mid$\hbox to 0pt{$\m@th#1^{\,#2}$}\hss}%
  \lower.9\ht0\hbox to 0pt{\hss$\m@th#1\smile$\hss}%
  \kern\wd0}

\def\dotminussym#1#2{%
  \setbox0=\hbox{$\m@th#1-$}%
  \kern.5\wd0%
  \hbox to 0pt{\hss\hbox{$\m@th#1-$}\hss}%
  \raise.6\ht0\hbox to 0pt{\hss$\m@th#1.$\hss}%
  \kern.5\wd0}
\newcommand{\dotminus}{\mathbin{\mathpalette\dotminussym{}}}

\def \d{\mathbf{D}}

\title{Oracle computability of conditional expectations onto subfactors}
\author{Isaac Goldbring}
\thanks{I. Goldbring was partially supported by NSF CAREER grant DMS-1349399.}

\address{Department of Mathematics\\University of California, Irvine, 340 Rowland Hall (Bldg.\# 400),
Irvine, CA 92697-3875}
\email{isaac@math.uci.edu}
\urladdr{http://www.math.uci.edu/~isaac}

\begin{document}

\begin{abstract}
We initiate the effective study of conditional expectations onto subfactors.  Our main result is that if $M$ is an existentially closed II$_1$ factor with a w-spectral gap subfactor $N$, then the conditional expectation function onto $N$ can be computed from a (Turing) oracle that computes a presentation of $M$, the inclusion of $N$ into $M$, and a spectral gap function for the pair $(M,N)$.
\end{abstract}

\maketitle

\section{Introduction}

Suppose that $M$ is a tracial von Neumann algebra and $N$ is a subalgebra of $M$.  In this paper, we consider the following question:  How ``hard'' is it to compute the conditional expectation function $E_N:M\to N$?  In other words, given $x\in M$, how ``hard'' is it to compute the distance from $x$ to $N$?  The main result of this paper gives an upper bound for the difficulty of computing $E_N$ in the case that $M$ is a separable \emph{existentially closed} II$_1$ factor and $N$ is a property (T) subfactor:

\begin{theorem}
Suppose that $M$ is a separable existentially closed II$_1$ factor with a property (T) subfactor $N$.  Suppose that we are given a presentation $M^\#$ of $M$ and a Kazhdan presentation $N^\dagger$ of $N$.  Suppose that $\d$ is a (Turing) oracle such that $M^\#$ is $\d$-computable and the inclusion $i:N^\dagger\to M^\#$ is $\d$-computable.  Then the conditional expectation map $E_N:M^\#\to N^\dagger$ is $\d$-computable.
\end{theorem}

While the terms in the previous theorem will be defined precisely throughout the paper, we now try to give some intuition as to what the theorem actually says.

We first say a few words about the class of existentially closed (e.c.) II$_1$ factors.  The notion of existentially closed object comes from model theory and is the model-theoretic generalization of the notion of algebraically closed field.  Roughly speaking, a II$_1$ factor $M$ is an e.c. factor if, whenever a system of equations with coefficients from $M$ has a ``solution'' in a II$_1$ factor extending $M$, then there is an approximate solution in $M$ itself.  Here, one can think of systems of equations as being described by moments in the free probability sense.  A purely ``semantic'' perspective on being existentially closed can be described using embeddings into ultrapowers and this will be the definition given in Section 4 below.  It is important to note that there is a large class of e.c. factors (whence the above result is not vacuous) although whether or not a concrete example of one can be given depends on the truth of the Connes Embedding Problem.

We now describe the remaining terms in the above theorem.  The result above is framed in the language of \emph{computability theory}.  In order to treat operator algebras in this framework, we need to encode a separable tracial von Neumann algebra using a countable amount of data.  Towards this end, we consider \emph{presentations} of II$_1$ factors, which are essentially dense countable *-subalgebras of the given von Neumann algebra that we imagine being able to input into a theoretical computer.  

In general, we do not expect that this theoretical computer should be able to (approximately) compute the algebraic operations and the 2-norm of the presentation.  Instead, our results are \emph{relative} in the sense that we allow computers with ``oracles'' meaning that the computers are allowed to also make queries to functions that are not in general computable.  (More on this in Section 2.) Thus, our main result says that if one has presentations of $M$ and $N$ such that one can compute the algebraic operations and 2-norm of the presentation of $M$ using some oracle $\d$ as well as how the presentation of $N$ ``sits'' inside of $M$, then this same oracle can, given an element $x$ in the presentation of $M$ and a rational tolerance $\epsilon$, compute an element of the presentation of $N$ that is within $\epsilon$ of $E_N(x)$.

Note that, in general, given an oracle that can compute the presentation of $M$ and an the inclusion of the presentation of $N$ in $M$, one can compute upper bounds for the distance of an element of $M$ to $N$ by simply calculating the distances between the element of $M$ and various elements of $N$ and recording every time one gets closer than previous calculations.  The nontrivial part is knowing when you cannot get any closer, namely, how do you know when to stop this procedure?  The import of our main theorem is that in the case of a property (T) subfactor of an e.c. factor, one can in fact determine when one has approximately achieved the minimum distance.

The argument proving the previous theorem is inspired by an argument of Macintyre \cite{Mac} on e.c. groups.  There, he proves that for each $n\geq 1$, there is a formula $\varphi(x_1,\ldots,x_n,y)$ in the language of groups such that, for any e.c. group $G$ and any $a_1,\ldots,a_n,b\in G$, one has that $b$ belongs to the subgroup of $G$ generated by $a_1,\ldots,a_n$ if and only if $\varphi(a_1,\ldots,a_n,b)$ is true in $G$.  In other words, while testing for membership in a subgroup is usually not a first-order concept, it actually is in e.c. groups and uniformly so over all e.c. groups.  Our original motivation was to try to adapt Macintyre's argument to the setting of e.c. factors (which seemed feasible as the main ingredients were amalgamated free products and HNN extensions, both of which make sense in the case of tracial von Neumann algebras) and the theorem above is what we obtained after adapting his arguments.

The formula $\varphi(x_1,\ldots,x_n,y)$ referred to in the previous paragraph is the existential statement asserting the existence of an element $u$ which commutes with the $x_i$'s yet does not commute with $y$.  (Clearly if such $u$ exists, then $y$ does not belong to the subgroup generated by $x_1,\ldots,x_n$; the miraculous part is that, in an e.c. group, the converse holds.)  If one tries to write down an analogous formula in the von Neumann algebra setting, then since the corresponding existence statement would be an approximate existence statement (as the existential quantifier is replaced by the approximate existential quantifier $\inf$), we would only know about an element $u$ which approximately commutes with the $x_i$'s.  In order to gain any traction, one would need to know that such a $u$ is near an element which actually commutes with the $x_i$'s.  It is this \emph{spectral gap} property of the property (T) subfactor $N$ that is needed to adapt Macintyre's work.  Thus, our main theorem is actually applicable in the wider context of spectral gap subfactors, the only issue being that the computability of a corresponding spectral gap function also becomes an extra parameter in the statement of the theorem.  This spectral gap function is automatically computable in the case of a property (T) subfactor using the existence of ``Kazdhan'' sets (a result due to Connes and Jones \cite{CJ}) and a presentation in which the Kazhdan set is computable.

In the next section, we briefly describe the basic facts from computability theory needed to understand the rest of the paper.  The two sections that follow describe the basic facts needed about spectral gap subfactors and existentially closed II$_1$ factors in order to understand the proof of the main theorem, which appears in the final section. 

We would like to thank Timothy McNicholl, Alexander Melnikov, and Andre Nies for patiently explaining the finer points of computable structure theory for metric structures.  We would also like to thank Adrian Ioana for useful discussions around spectral gap and property (T).


\section{Computable structure theory for II$_1$ factors}

\subsection{Turing oracles}

We begin this section with a brief discussion of the basic notions from computability theory; a good and accessible reference is \cite{Enderton}.

Computability theory studies the question of what it means for a function $f:\mathbb N^k\to \mathbb N$ to be ``computable.''  Na\"ively speaking, such a function $f$ should be computable if there is an algorithm\footnote{This algorithm is allowed to refer to basic arithmetic operations such as addition and multiplication.} such that, upon input $(a_1,\ldots,a_k)\in \mathbb N^k$, runs and eventually halts, outputting the result $f(a_1,\ldots,a_k)$.  There are many approaches to formalizing this heuristic (e.g. Turing-machine computable functions and recursive functions) and all known formalizations can be proven to yield the same class of functions.  This latter fact gives credence to the \emph{Church-Turing thesis}, which states that this aformentioned class of functions is indeed the class of functions that are computable in the na\"ive sense described above.  In the rest of this paper, we will never argue about this class of functions using any formal definition but will only argue informally in terms of some kind of algorithm or computer; this is often referred to as arguing using the Church-Turing Thesis.  

As discussed in the introduction, the functions that we are interested in are almost certainly not computable in the sense of the previous paragraph.  Instead, we will consider a relative form of computability in the following sense:  given a function $\d:\mathbb N^l\to \mathbb N$, a \textbf{$\d$-algorithm} is an algorithm in the usual sense that is also allowed to make ``queries'' to an ``oracle'' which has access to the function $\d$.  Thus, for example, in the course of running the algorithm, the computer is allowed to query the oracle and ask what is the value of $\d(15)$ and then use that information in the computation.  A function $f$ as above is then said to be \textbf{$\d$-computable} (or that \textbf{$\d$ computes $f$}) if it is computable using a $\d$-algorithm.  Of course, if $\d$ were itself computable, then $f$ would be computable as well.

In the rest of this paper, we often write statements such as ``Suppose that $\d$ is an oracle that computes the functions $f_1$ and $f_2$.''  In practice, one would often know that each $f_i$ is $\d_i$-computable for some oracles $\d_1$ and $\d_2$.  However, standard coding tricks allow us to find a single oracle $\d$ that can compute functions that both $\d_1$ and $\d_2$ can compute (and there is a least such $\d$ in a precise sense).  Consequently, we find it cleaner to make statements whose assumptions posit the existence of a single oracle that can compute a handful of functions


\subsection{Computable structure theory}

In this subsection, we present the basic definitions from computable structure theory for metric structures \cite{FMcN} (building upon the work in \cite{CS1},\cite{CS2}, \cite{CS3}, \cite{CS4}, \cite{CS5}) in the context of tracial von Neumann algebras.

Throughout this subsection, $M$ denotes a separable tracial von Neumann algebra whose unit ball is denoted by $M_1$.  Given $x,y\in M_1$, a \textbf{rounded combination} of $x$ and $y$ is an element of the form $\lambda x+\mu y$, where $\lambda,\mu\in \mathbb C$ satisfy $|\lambda|+|\mu|\leq 1$.  The rounded combination will be called \textbf{rational} if $\lambda$ and $\mu$ belong to $\mathbb Q(i)$.   

\begin{defn}

\

\begin{enumerate}
\item Given $A\subseteq M_1$, we let $\langle A\rangle$ be the smallest subset of $M_1$ containing $A$ and closed under rational rounded combinations, multiplication, and adjoint.\footnote{In order to fit our discussion under the more general presentation found in \cite{FMcN}, we need our operations to be uniformly continuous, whence the need to restriction attention to operator norm bounded balls.}
\item We say that $A$ \textbf{generates} $M$ if $\langle A\rangle$ is 2-norm dense in $M_1$.
\item A \textbf{presentation} of $M$ is a pair $M^\#:=(M,(a_n)_{n\in \mathbb N})$, where $\{a_n \ : \ n\in \mathbb N\}\subseteq M_1$ generates $M$.  Elements of the sequence $(a_n)_{n\in \mathbb N}$ are referred to as \textbf{special points} of the presentation while elements of $\langle \{a_n \ : \ n\in \mathbb N\}\rangle$ are referred to as \textbf{rational points} of the presentation.
\end{enumerate}
\end{defn}


The following remark is crucial for what follows:

\begin{rmk}
Given a presentation $M^\#$ of $M$, it is possible to computably enumerate the rational points of $M^\#$.\footnote{Technically, one is computably enumerating ``codes'' for rational points.}  Consequently, it makes sense to consider algorithms which take rational points of $M^\#$ as inputs and/or outputs.
\end{rmk}

\begin{defn}
If $M^\#$ is a presentation of $M$ and $\d$ is an oracle, then $x\in M_1$ is a \textbf{$\d$-computable point of $M^\#$} if there is a $\d$-algorithm such that, upon input $k\in \mathbb N$, returns a rational point $p\in M^\#$ with $d(x,p)<2^{-k}$.
\end{defn}

\begin{defn}
If $M^\#$ is a presentation of $M$ and $\d$ is an oracle, then $M^\#$ is a \textbf{$\d$-computable presentation} if there is a $\d$-algorithm such that, upon input rational point $p\in M^\#$ and $k\in \mathbb N$, returns a rational number $q$ such that $|p\|_2-q|<2^{-k}$.
\end{defn}



\begin{defn}
Suppose that $M$ and $N$ are tracial von Neumann algebras with presentations $M^\#$ and $N^\dagger$ respectively.  Further suppose that $f:M^m\to N$ is a Lipshitz map (say with respect to the maximum metric on $M^m$)\footnote{The Lipshitz condition can be weakened to having a ``computable modulus of uniform continuity'' but we will not need this more general notion in this paper.} and $\d$ is an oracle.  Then $f$ is a \textbf{$\d$-computable map from $M^\#$ into $N^\dagger$} if there is a $\d$-algorithm such that, upon input a tuple of rational points $\vec p\in (M^\#)^m$ and $k\in \mathbb N$, returns a rational point $p'\in N^\dagger$ such that $d(f(\vec p),p')<2^{-k}$.  (Here, we use an efficient numbering of $\mathbb N^k$ to effectively enumerate the $m$-tuples of rational points of $M^\#$.)
\end{defn}


\subsection{Conditional expectations}

In this subsection, we make a few general observations about computability of conditional expectations.  First, we introduce a convenient piece of terminology:

\begin{defn}
Suppose that $N$ is a subfactor of $M$ and that $M^\#$ and $N^\dagger$ are presentations of $M$ and $N$ respectively.  For an oracle $\d$, we say that $(M^\#,N^\dagger)$ is a \textbf{$\d$-computable pair} if $M^\#$ is a $\d$-computable presentation of $M$ and the inclusion map $i:N^\dagger\to M^\#$ is a $\d$-computable map.
\end{defn}

\begin{rmk}
If $(M^\#,N^\dagger)$ is a $\d$-computable pair, then $N^\dagger$ is a $\d$-computable presentation of $N$.
\end{rmk}

\begin{lemma}
Suppose that $(M^\#,N^\dagger)$ is a $\d$-computable pair.  Then $E_N:M^\#\to N^\dagger$ is $\d$-computable if and only if there is a $\d$-algorithm which, upon input rational point $p\in M^\#$ and $k\in \mathbb N$, produces a rational number $q$ such that $|d(p,N)-q|<2^{-k}$.
\end{lemma}

\begin{proof}
First suppose that $E_N:M^\#\to N^\dagger$ is $\d$-computable.  Fix a rational point $p\in M^\#$ and $k\in \mathbb N$.  Using $\d$, we can find a rational point $p'\in N^\dagger$ such that $d(E_N(p),p')<2^{-k-1}$.  Since $i:N^\dagger\to M^\#$ is $\d$-computable, using $\d$ we can find a rational point $p''\in M^\#$ such that $d(i(p'),p'')<2^{-k-1}$.  Since $$|d(p,N)-d(p,p'')|=|d(p,E_N(p))-d(p,p'')|\leq d(E_N(p),p')+d(i(p'),p'')<2^{-k},$$ and one can compute $d(p,p'')$ using $\d$, this $\d$-algorithm computes $d(p,N)$. 

We now prove the converse.  Suppose that $p\in M^\#$ is a rational point and $k\in \mathbb N$.  Set $l:=k-2$.  By the hypothesis, using $\d$ we can find a rational number $q$ such that $|d(p,N)-q|<2^{-l}$.  Now, using $\d$ again, start computing $d(p,i(p'))$ to within $2^{-l}$ for rational points $p'\in N^\#$.  Suppose that $p'\in N^\dagger$ is the first rational point for which there is a rational number $r$ such that $d(p,i(p'))$ is within $2^{-l}$ of $r$ and $(q-2^{-l},q+2^{-l})\cap (r-2^{-l},r+2^{-l})\not=\emptyset$.  Then since 
$$\|p-p'\|_2^2=\|(p-E_N(p))+(E_N(p)-p')\|_2^2=\|p-E_N(p)\|_2^2+\|E_N(p)-p'\|_2^2,$$ we have $d(E_N(p),p')^2=d(p,p')^2-d(p,E_N(p))^2\leq 6\cdot 2^{-l}+(3\cdot 2^{-l})^2<2^{-l+3}+2^{4-2l}<2^{-l+2}=2^{-k}$.  This algorithm thus computes $E_N$.
\end{proof}

We end this section with a brief comment about the finite index case.

\begin{prop}
Suppose that $N$ has finite index in $M$, that $(M^\#,N^\dagger)$ is a $\d$-computable pair, and that there is a Pimsner-Popa basis $m_1,\ldots,m_{n+1}$ for $M$ over $N$ (see \cite{pimsnerpopa}) such that each $m_i$ is a $\d$-computable point of $M^\#$ and each $E_N(m_j)$ is a $\d$-computable point of $N^\dagger$.  Then $E_N$ is $\d$-computable.
\end{prop}

\begin{proof}
Given a rational point $p\in M^\#$ and $k\in \mathbb N$, one first computes $\sum m_jp_j$ (as $p_j$ range over the presentation of $N$) and waits until it is within $2^{-k}$ of $x$.  Then since $E_N$ is 2-norm contractive, $E_N(x)$ is with $2^{-k}$ of $\sum E_N(m_j)p_j$, which we can also compute using $\mathbf{d}$.
\end{proof}

\begin{question}
If $(M^\#,N^\dagger)$ is a $\d$-computable pair of finite index, must a Pimsner-Popa basis as above always exist?
\end{question}

\section{w-spectral gap}

Throughout this section, $M$ is a II$_1$ factor and $N$ is a subfactor of $M$.  We remind the reader of the definition of w-spectral gap.

\begin{defn}
$N$ has \textbf{w-spectral gap in $M$} if, for any $\epsilon>0$, there is a finite $F\subseteq N$ and $\delta>0$ such that, for all $x\in M_1$, we have:  if $\max_{y\in F}\|[x,y]\|_2<\delta$, then there is $x'\in N'\cap M$ such that $d(x,x')<\epsilon$.
\end{defn}

The reader is invited to consult \cite{Gold} for more information about w-spectral gap subfactors and their model-theoretic significance.  In order to bring this notion into our computability-theoretic setting, we need the following definition:

\begin{defn}
Suppose that $N$ is a w-spectral gap subfactor of $M$ and that $M^\#$ and $N^\dagger=(N;(a_n)_{n\in \mathbb N})$ are presentations of $M$ and $N$ respectively.  We say that $f:\mathbb N\to \mathbb N$ is a \textbf{spectral gap function for $(M^\#,N^\dagger)$} if:  for any $n\in \mathbb N$ and rational point $p\in M^\#$, if $\max_{1\leq i\leq f(n)}\|[p,a_i]\|_2<2^{-f(n)}$, then $d(p,N'\cap M)<2^{-n}$.
\end{defn}

\begin{rmk}
If $N$ has w-spectral gap in $M$, then for any presentations $M^\#$ and $N^\dagger$ of $M$ and $N$ respectively, there is a spectral gap function $f$ for $(M^\#,N^\dagger)$.
\end{rmk}

In order to obtain our results about property (T) subfactors as a special case of our main result on w-spectral gap subfactors, we need to remind the reader of the following fact of Connes and Jones \cite[Proposition 1]{CJ}:

\begin{fact}\label{kazhdan}
Suppose that $N$ is a II$_1$ factor with property (T).  Then there is $\epsilon>0$, finite $F\subseteq M$, and $K>0$ such that, for any $\delta\leq \epsilon$, any $N$-$N$ bimodule $H$, and any unit vector $\xi\in H$, if $\|y\xi-\xi y\|<\delta$ for all $y\in F$, then there is a central vector $\eta\in H$ such that $\|\eta-\xi\|<K\delta$.  
\end{fact}

We refer to the finite set $F$ in the previous fact as a \textbf{Kazhdan set} for $N$ and the pair $(F,K)$ as a \textbf{Kazhdan pair} for $N$ (in analogy with the corresponding terminology for groups).  We call a presentation $N^\dagger$ of a property (T) factor $N$ a \textbf{Kazhdan presentation} if there is a Kazhdan set for $N$ amongst the rational points of $N^\dagger$.

Before stating the main computability-theoretic fact about Kazhdan presentations of property (T) factors, we state one easy lemma, whose proof we leave to the reader.

\begin{lemma}
There is a computable function $j:\mathbb N^2\to \mathbb N$ such that, for any tracial von Neumann algebra $N$, any presentation $N^\dagger$ of $N$, and any tracial von Neumann algebra $M$ containing $N$, if $p$ is the $m^{\text{th}}$ rational point of $N^\dagger$ and $x\in M_1$ is such that $\max_{1\leq i \leq j(m,k)}\|[x,a_i]\|_2<2^{-j(m,k)}$, then $\|[x,p]\|_2<2^{-k}$.
\end{lemma}

\begin{cor}
Suppose that $N$ is a property (T) factor and $N^\dagger$ is a Kazhdan presentation of $N$.  Then for any II$_1$ factor containing $M$ and any presentation $M^\#$ of $M$, there is a computable spectral gap function for $(M^\#,N^\dagger)$.
\end{cor}

\begin{proof}
Take a finite subset $F$ of $M_1$ and $p\in \mathbb N$ is such that $(F,2^p)$ is a Kazhdair pair for $N$ as witnessed by $\epsilon:=2^{-p}$ as in the statement of Fact \ref{kazhdan} and for which there is $m\in \mathbb N$ such that $F$ is contained amongst the first $m$ rational points of $N^\dagger$.  Set $$f(n):=\max_{1\leq i\leq m}j(i,n+p).$$ Note that $f$ is a computable function.  We claim that $f$ is a spectral gap function for $(M^\#,N^\dagger)$.  Indeed, suppose that $b$ is a rational point of $M^\#$ and $\max_{1\leq i\leq f(n)}\|[b,a_i]\|_2<2^{-f(n)}$.  Then by the definition of $j$, we have $\max_{p\in F}\|[b,p]\|_2<2^{-n-p}$, whence by Fact \ref{kazhdan} above, there is $p'\in N'\cap M$ such that $d(p,p')<2^{-n}$, as desired.
\end{proof}

\begin{rmk}
In the previous corollary, one can remove the assumption that the presentation is a Kazhdan presentation at the cost of concluding that the spectral gap function is merely $\d$-computable, where $\d$ is some oracle for which each member of some Kazhdan set for $N$ is a $\d$-computable point.
\end{rmk}

%


\section{Existentially closed II$_1$ factors}

In this brief section, we remind the reader of the definition of existentially closed II$_1$ factor and mention a few remarks about them.

\begin{defn}
A II$_1$ factor $M$ is \textbf{existentially closed} (or e.c. for short) if:  whenever $P$ is a II$_1$ factor such that $M\subseteq P$, then there is an ultrafilter $\mathcal U$ and an embedding $i:P\hookrightarrow M^\mathcal U$ such that the restriction $i|M$ is the diagonal embedding of $M$ in $M^\mathcal U$.
\end{defn}

The class of e.c. II$_1$ factors is an incredibly rich family of factors (see \cite{ecfactor} for more details).  Every II$_1$ factor embeds into an e.c. factor (of the same density character).  The hyperfinite II$_1$ factor $\mathcal R$ is an e.c. factor if and only if the Connes Embedding Problem has a positive solution.  w-spectral gap subfactors of e.c. factors were studied in \cite{Gold}, where it was shown, in particular, that if $N$ is a w-spectral gap subfactor of the e.c. factor $M$, then $(N'\cap M)'\cap M=N$.

We note that w-spectral gap subfactors of e.c. factors have infinite index, whence do not fall into the discussion from Section 2 above.  Indeed, since w-spectral gap subfactors never have property Gamma\footnote{This is probably well-known, but here is a proof communicated to us by Adrian Ioana:  Suppose that $N$ is a w-spectral gap subfactor of $M$.  Fix $\epsilon>0$ and take a finite set $F=F(\epsilon)$ witnessing w-spectral gap.  Suppose that $u\in U(N)$ $\epsilon$-commutes with $F$.  Then there is $x\in N'\cap M$ such that $d(u,x)\leq \epsilon$.  Since $E_N(x)\in Z(N)=\mathbb C\cdot 1$, it follows that $d(u,\mathbb C)\leq \epsilon$.  If $\epsilon$ is sufficiently small, this prevents $\operatorname{tr}(u)=0$.} while e.c. factors always have property Gamma (they are in fact McDuff \cite[page 3]{nomodcomp}), it follows that w-spectral gap subfactors of e.c. factors always have infinite index by \cite[Proposition 1.11]{pimsnerpopa}.

%
%
%

\section{Proof of the Main Result}

In this section, we prove the main result announced in the introduction.  First, given $m\in \mathbb N$, a II$_1$ factor $M$, and $u,a_1,\ldots,a_n,b\in M$, we set

$$\psi^M_{r,m}(u,\vec a,b):=\max\left(\|uu^*-1\|_2,\max_{1\leq i\leq m}\|[u,a_i]\|_2,2r\dotminus\|[u,y]\|_2\right).$$

Here, $\dotminus$ denotes truncated subtraction, that is, $r\dotminus s:=\max(r-s,0)$.  We also set 
$$\varphi^M_{r,m}(\vec a,b):=\inf_{u\in M_1}\psi^M_{r,m}(u,\vec a,b).$$

%

In the next lemma (which follows almost immediately from the definitions), we view $\mathbb C$ as a tracial von Neumann algebra in the obvious way and consider its ``standard'' presentation, that is, with presentation $\mathbb Q(i)$ (enumerated in some computable fashion).

\begin{lemma}\label{cancalculate}
Suppose that $M^\#$ is $\d$-computable.  Then for each rational number $r$ and each $m\in \mathbb N$, $\psi^M_{r,m}:(M^\#)^{m+2}\to \mathbb C$ is $\d$-computable.
\end{lemma}

The following lemma follows from a standard functional calculus argument:

\begin{lemma}
There is a nondecreasing computable function $g:\mathbb N\to \mathbb N$ such that, for all $k\in \mathbb N$, all tracial von Neumann algebras $M$, and all $u\in M_1$, if $\|uu^*-1\|_2<2^{-g(k)}$, then there is a unitary $u'\in M$ such that $d(u,u')<2^{-k}$.
\end{lemma}


\begin{lemma}\label{helpful}
Suppose that $f:\mathbb N\to \mathbb N$ is a spectral gap function for $(M,N^\#)$.  Set $f':\mathbb N\to \mathbb N$ to be $f'(n):=f(g(n+2)+1)$.  Then for any $n$ and $r>0$, if $\varphi^M_{r,f'(n)}(x_1,\ldots,x_{f'(n)},b)<2^{-f'(n)}$, then $d(b,N)\geq r-2^{-n}$.
\end{lemma}

\begin{proof}
Suppose $u\in M_1$ is such that  $\psi^M_{r,f'(n)}(u,x_1,\ldots,x_{f'(n)},b)<2^{-f'(n)}$.  Then there is $u_1\in N'\cap M$ such that $d(u,u_1)<2^{-g(n+2)-1}$. Note then that $\|u_1u_1^*-1\|_2<2^{-g(n+2)}$, whence there is $u_2\in U(N'\cap M)$ such that $d(u_1,u_2)<2^{-(n+2)}$.  It follows that
$$2r\dotminus \|[b,u_2\|\leq 2^{-f'(n)}+2^{-g(n+2)-1}+2^{-(n+2)}<2^{-n}.$$

If $s=d(b,N)$, we have elements $b_n\in N$ such that $d(b,b_n)\to s$; since $d(u_2bu_2^*,b_n)=d(b,b_n)$ (as $u_2b_nu_2^*=b_n$),. we have $d(b,u_2bu_2^*)\leq 2s$.  It follows that $2r-2\cdot 2^{-n}\leq \|[b,u_2]\|_2\leq 2s$, so $s\geq r-2^{-n}$.
\end{proof}

\begin{rmk}
Since $g$ is computable, $f'$ is $\d$-computable if $f$ is $\d$-computable.
\end{rmk}

%
%

\begin{lemma}\label{check}
Suppose that $M$ is e.c., that $N$ is a subfactor of $M$, and that $(N,(x_n)_{n\in \mathbb N})$ is a presentation of $N$.  If $b\in M$ is such that $d(b,N)=r$, then $\varphi_{r,m}^M(\vec x,b)=0$ for all $m\in \mathbb N$. 
\end{lemma}

\begin{proof}
Let $P$ be the subalgebra of $M$ generated by $N$ and $b$.  Set $M_1=M*_{N}P$ and let $b'$ denote the other copy of $b$ in $M_1$.  Note that $d(b,b')=2r$.  Let $\theta:P\hookrightarrow M_1$ be given by $\theta|N=\operatorname{id}_N$ and $\theta(b)=b'$.  Let $M_2$ be the HNN extension of $M_1$ with respect to the embedding $\theta$.  (See \cite{HNN} for details on HNN extensiosn of von Neumann algebras.). In particular, $M_2$ is a II$_1$ factor with a unitary element $u$ so that $\|[u,a_i]\|_2=0$ for $i=1,\ldots,m$ while $\|[u,b]\|_2=d(b,b')=2r$, whence $\varphi^{M_2}(\vec x,b)=0$.  Since $M$ is e.c., $\varphi^M(\vec x,b)=0$, as desired.
\end{proof}

Here is the main theorem of this paper:

\begin{thm}
Suppose that $N$ is a w-spectral gap subfactor of the e.c. factor $M$ and that $M$ and $N$ have presentations $M^\#$ and $N^\dagger$ respectively so that the pair $(M^\#,N^\dagger)$ is a $\d$-computable pair for some oracle $\d$.  Further suppose that there is a $\d$-computable spectral gap function for $(M^\#,N^\dagger)$.  Then $E_N:M^\#\to N^\dagger$ is $\d$-computable.
\end{thm}

\begin{proof}
Suppose we are given a rational point $p$ of $M^\#$ and $k\in \mathbb N$.  We use $\d$ to compute $d(p,N^\#)$ to within $2^{-k}$ using two machines.

On the first machine, we start approximately computing $d(i(p'),p)$, where $p'$ ranges over rational points of $N^\dagger$.  This is done by finding rational $p''\in M^\#$ such that $d(i(p'),p'')$ is small, and then computing $d(p'',p)$ approximately.  This machine thus enumerates upper bounds for $d(p,N^\#)$.

We next use $\d$ to compute $f'(k)$.  On the second machine, we start computing $\psi_{r,f'(k)}(u,\vec p,b)$ for rational points $u$ and rational numbers $r$; by Lemma \ref{cancalculate}, this can be done using $\d$.  If we see $\psi^M_{r,f'(k)}(u,\vec p,b)<2^{-f'(k)}$, then we know that $d(b,N)\geq r-2^{-k}$ by Lemma \ref{helpful}.

We then wait until there is a rational number $r>0$ so that the first machine tells us that $d(b,N)<r$ and the second machine tells us that $d(b,N)\geq r-2^{-k}$.  By Lemma \ref{check}, this is guaranteed to happen.
\end{proof}

\begin{cor}
Suppose that $M$ is e.c. and $N$ is a property (T) subfactor of $M$.  Let $N^\dagger$ be a Kazhdan presentation of $N$.  If $(M^\#,N^\dagger)$ is a $\d$-computable pair, then $E_N$ is $\d$-computable.
\end{cor}

\end{document}